\documentclass[a4paper, 12pt]{article}

\usepackage[sort&compress]{natbib}
\bibpunct{(}{)}{;}{a}{}{,} 

\usepackage{amsthm, amsmath, amssymb, mathrsfs, multirow, url, subfigure}
\usepackage{setspace}
\usepackage{graphicx} 
\usepackage{ifthen} 
\usepackage{amsfonts}
\usepackage[usenames]{color}
\usepackage{fullpage}

\theoremstyle{plain}

\newtheorem{prop}{Proposition}

\theoremstyle{definition}

\theoremstyle{remark}

\newcommand{\prob}{\mathsf{P}} 
\newcommand{\E}{\mathsf{E}}

\newcommand{\bel}{\mathsf{bel}}
\newcommand{\pl}{\mathsf{pl}}
\newcommand{\C}{\mathscr{C}}
\renewcommand{\P}{\mathscr{P}}

\newcommand{\unif}{{\sf Unif}}
\newcommand{\nm}{{\sf N}}

\newcommand{\chisq}{{\sf ChiSq}}

\newcommand{\K}{\mathcal{K}}
\newcommand{\RR}{\mathbb{R}}
 
\newcommand{\U}{\mathscr{U}}

\newcommand{\YY}{\mathbb{Y}}
\newcommand{\UU}{\mathbb{U}}

\renewcommand{\S}{\mathcal{S}}
\renewcommand{\SS}{\mathbb{S}}

\renewcommand{\phi}{\varphi}

\title{Random sets and exact confidence regions}
\author{
Ryan Martin \\
Department of Mathematics, Statistics, and Computer Science \\
University of Illinois at Chicago \\
\url{rgmartin@uic.edu} 
}
\date{\today}

\begin{document}

\maketitle 


\begin{abstract}
An important problem in statistics is the construction of confidence regions for unknown parameters.  In most cases, asymptotic distribution theory is used to construct confidence regions, so any coverage probability claims only hold approximately, for large samples.  This paper describes a new approach, using random sets, which allows users to construct exact confidence regions without appeal to asymptotic theory.  In particular, if the user-specified random set satisfies a certain validity property, confidence regions obtained by thresholding the induced data-dependent plausibility function are shown to have the desired coverage probability.  

\smallskip


\smallskip

\end{abstract}

\section{Introduction}
\label{S:intro}

A fundamental problem in statistics is that of constructing confidence regions.  Roughly speaking, a confidence region is a data-dependent subset of the parameter space with the interpretation that, all values inside this subset are ``reasonable'' estimates of the unknown parameter.  The more precise interpretation of confidence regions is based a frequentist notion of coverage probability.  That is, in repeated sampling, the confidence region will contain the true parameter value a specified proportion of the time.  That the confidence region (nearly) hits the target coverage probability is crucial to the validity of the resulting inference: on one hand, if the actual coverage probability is too high, then the confidence regions are likely too large to provide any meaningful notion of uncertainty; on the other hand, if the actual coverage probability is too low, then it is likely that the confidence region has a systematic bias, casting doubt on the accuracy of the results for the data at hand.  Unfortunately, it is rare that a simple and exact confidence region is available; the well-known Student-t confidence interval for a Gaussian mean is one exception.  Typically, an appeal to asymptotic theory is made, and confidence regions are built based on the simpler limiting distribution; confidence regions based on the asymptotic normality of maximum likelihood estimators is one example.  However, with this approach, one must add to any coverage probability claim the caveat ``for sufficiently large sample size.''  Alternatively, numerical methods, such as bootstrap \citep{efrontibshirani1993}, are popular when a direct appeal to asymptotic theory is questionable.  But validity of the bootstrap also depends on large-sample theory, so there are no non-asymptotic coverage probability guarantees for bootstrap confidence regions. 

This paper describes a new approach, resting on a theory of random sets.  The initial step is to establish an association between the observable data, the unknown parameter, and a mostly arbitrary auxiliary variable.  By ``association'' here we mean a suitable representation of the statistical model for the observable data.  Alternatively, the association can be viewed as a sort of compatibility relation among the various inputs. Random sets supported in the auxiliary variable space---called predictive random sets---are propagated, via observed data and the specified statistical model, to random sets in the parameter space.  These random sets in the parameter space are characterized by their belief functions or, alternatively, by their plausibility functions.  These functions also appear in the famous Dempster--Shafer theory \citep{dempster2008, shafer1976}, but the approach described here is different.  It is shown in Section~\ref{S:plausibility} that, under very mild conditions on the user-specified predictive random sets, exact confidence regions can be constructed via a suitable thresholding of the plausibility function.  Moreover, the plausibility regions have a simple but inferentially meaningful interpretation.  

The remainder of the paper is organized as follows.  Section~\ref{S:setup} describes the general statistical problem and defines confidence regions and coverage probability.  Random sets are described in Section~\ref{S:sets}, with a general overview in Section~\ref{SS:overview} and a presentation of the important new concept, namely, predictive random sets, in Section~\ref{SS:prs}.  These sets are the driving force behind the proposed approach.  In Section~\ref{S:plausibility} we first define the plausibility function, which is nothing but a probability calculation relative to the distribution of the predictive random set, along with the corresponding plausibility region.  Then we prove the main result that, under mild conditions on the model itself, if the predictive random set is valid, a property that is easily satisfied, then the corresponding plausibility regions hit the desired coverage probability.  This is a finite sample result, not asymptotic.  Here we find that certain aspects of the formal mathematical theory of random sets leads to a relatively simple statement of the sufficient conditions for this result.  Three illustrative examples, involving models used in reliability theory, are presented in Section~\ref{S:examples}.  Finally, Section~\ref{S:discuss} contains a brief discussion.

\section{Setup and notation}
\label{S:setup}

Let $Y$ be an observable sample, taking values in the sample space $\YY$, with distribution $\prob_{Y|\theta}$ depending on a parameter $\theta$ in a parameter space $\Theta$.  Here $Y$ may be a vector of $n$ (possibly independent) observations, so that $\YY$ is actually a product space, but it is not necessary to be so specific here.  The distribution $\prob_{Y|\theta}$ is called the sampling model, and if the value of $\theta$ were known, then $Y$ could be simulated.  In the present context, the actual $\theta$ is unknown, and the goal is to use data $Y$ to make inference about $\theta$.  

In statistical applications, it is typical to summarize data $Y$ with a statistic $T=T(Y)$.  Just like $Y$, the statistic $T$ has a sampling distribution, denoted by $\prob_{T|\theta}$, which usually depends on $\theta$.  In fact, one usually takes $T$ to be a minimal sufficient statistic for $\theta$, though deviation from this guiding principle is sometimes warranted; see Section~\ref{SS:power}.  The initial reduction of $X$ to a minimal sufficient statistic $T$ can be justified by the standard arguments of Fisher or, more generally, by those of \citet{imcond}.  The classical frequentist approach to statistical inference derives procedures, such as hypothesis tests, based on the sampling distribution of $T$.  In this paper, focus is on confidence regions.  Let $\C_\alpha(T)$ be a $T$-dependent subset of $\Theta$.  For given $\alpha \in (0,1)$, $\C_\alpha(T)$ is called a $100(1-\alpha)$\% confidence region for $\theta$ if 
\begin{equation}
\label{eq:coverage}
\prob_{T|\theta}\{\C_\alpha(T) \ni \theta\} \geq 1-\alpha, \quad \forall \; \theta \in \Theta. 
\end{equation}
The left-hand side of \eqref{eq:coverage} is the coverage probability of $\C_\alpha(T)$, and the definition of confidence region places a condition on this coverage probability, namely, that it must exceed the $1-\alpha$ level.  In words, \eqref{eq:coverage} states that, if the confidence region $\C_\alpha(T)$ is used in many examples involving data $Y \sim \prob_{Y|\theta}$ and statistic $T=T(Y)$, then roughly $100(1-\alpha)$\% of the realized regions will contain the true parameter value.  In other words, if $\alpha$ is small, i.e., $\alpha=0.05$, then $\{\C_\alpha(T) \not\ni \theta\}$ is a rare event with respect to the sampling distribution of $T$.  So, in practice, users will use this ``rare event'' interpretation to justify the conclusion that their calculated confidence region contains the true $\theta$ value.  

Clearly, it is most efficient for the $100(1-\alpha)$\% confidence region to have coverage probability equal $1-\alpha$; this would indicate that, in some sense, its size is just right.  In practice, however, for the sake of analytical or computational convenience, this efficiency is sacrificed.  That is, confidence regions used in practice may not exactly satisfy \eqref{eq:coverage}.  Equality may hold in \eqref{eq:coverage} only as $n \to \infty$, and for finite $n$, the true coverage probability may be above or below the desired $1-\alpha$ level.  It would desirable to have a general way to construct regions $\C_\alpha(T)$ that satisfy \eqref{eq:coverage} for all $n$, especially if equality can be attained in some cases.  The objective of this note is to present and justify such a construction.  

Towards this, we must first digress a bit to introduce an alternative representation of the sampling model $\prob_{T|\theta}$, one that involves an auxiliary variable.  Let $\UU$ be an (arbitrary) auxiliary variable space, equipped with a probability measure $\prob_U$.  Then choose a function $a: \UU \times \Theta \to \Theta$, such that, if $U \sim \prob_U$, then $a(U,\theta) \sim \prob_{T|\theta}$.  In other words, the sampling distribution of $T$ can be characterized by the following recipe:
\begin{equation}
\label{eq:sample}
\text{sample $U \sim \prob_U$ and set $T=a(U,\theta)$}. 
\end{equation}
This is a familiar notion in the context of simulation, e.g., the inverse probability transform, etc, but here the motivation is different.  The function $a$ forges an association between data $T$, parameter $\theta$, and auxiliary variable $U$.  The point is that, once $T=t$ is observed, the very best possible inference about $\theta$ is obtained if and only if the corresponding $U$ value is observed.  As $U$ is, by construction, unobservable, the inference problem can be recast into one of accurately guessing or predicting the unobserved $U$.  This is where random sets will come in handy.

\section{Random sets}
\label{S:sets}

\subsection{A general overview}
\label{SS:overview}

Let $\UU$ be a space and $\S$ a random set, taking values in a collection of subsets of $\UU$, with distribution $\prob_\S$.  There is a rigorous theory for random sets, presented beautifully in, e.g., \citet{molchanov2005}.  Our case here turns out to be a relatively simple special case of the general theory so, for the sake of simplicity, we shall mostly ignore the topological and measure-theoretical technicalities that appear in more formal treatments.  

There are several ways to describe the distribution of the random set.  One approach is via the the plausibility function $\pl(K) = \prob_\S\{\S \cap K \neq \varnothing\}$, $K \subseteq \UU$, also called the capacity functional \citep[][Def.~1.4]{molchanov2005}.  An equivalent quantity, often used in artificial intelligence applications, is the belief function, $\bel(K) = 1-\pl(K^c)$.  The belief function is a formal analogue to the distribution function of a random variable.  One key difference when dealing with random sets, compared to random variables, is that the complementation law generally fails, i.e., $\pl(K) + \pl(K^c) \geq 1$, with equality for all $K$ if and only if $\pl$ is a probability measure if and only if $\S$ is a singleton set with $\prob_\S$-probability~1.  One can discuss belief and plausibility functions without explicitly talking about random sets \citep[e.g.,][]{shafer1979}, though we shall not do so here.  One particularly natural way that random sets and plausibility functions emerge is through a sort of push-forward operation on a probability measure via a set-valued mapping; see, e.g., \citet{dempster1967}, \citet{nguyen1978}, and the discussion following the proof of Proposition~\ref{prop:prs.valid} below.  There is now a wide variety theoretical developments and applications of plausibility functions; see the volume edited by \citet{yagerliu2008}.  In the remainder of this section, we shall focus only on those details that will be important in the sequel.  

Consider now the special case where the random set is nested.  In other words, the collection $\SS \subset 2^\UU$ of possible realizations of $\S$, called the support of $\S$, satisfies:
\begin{equation}
\label{eq:nested}
\text{for any $S,S' \in \SS$, either $S \subseteq S'$ or $S \supseteq S'$}. 
\end{equation}
In this case, the plausibility function corresponding to $\S$ is called \emph{consonant} \citep{shafer1976, shafer1987, aregui.denoeux.2008, balch2012}.  In this case, it is easy to see that 
\[ \pl(K_1 \cap K_2) = \max\{\pl(K_1), \pl(K_2)\}, \quad K_1,K_2 \subseteq \UU. \]
This immediately extends to finite intersections and, with some standard continuity assumptions on $\pl$, to countable intersections.  Here we shall consider a stronger version of continuity, called \emph{condensability} \citep[e.g.,][]{shafer1976, shafer1987, nguyen1978}.  That is, the plausibility function is condensable if, for any upward net $\K$ of subsets in $\UU$, 
\begin{equation}
\label{eq:condense}
\pl\Bigl( \bigcup_{K \in \K} K \Bigr) = \sup_{K \in \K} \pl(K); 
\end{equation}
by an upward net we mean a collection of sets with the property that, for any $K_1$ and $K_2$ in $\K$, there exists $K_3$ in $\K$ such that $K_3 \supset K_1 \cup K_2$.  That this generalizes the usual notion of continuity of set functions is clear.  

Consonance and condensability together imply that the plausibility function is fully characterized \citep[see][Sec.~V.G]{shafer1987} by the contour function, given by
\begin{equation}
\label{eq:contour}
f_\S(u) = \prob_\S\{ \S \ni u \}, \quad u \in \UU, 
\end{equation}
i.e., the probability that the random set $\S$ catches the fixed point $u \in \UU$.  As this is an ordinary function, not a set function, it will be easier to work with than the full plausibility function.  That this function captures the entire plausibility function can be seen by the formula $\pl(K) = \sup_{u \in K} f_\S(u)$, a special case of \eqref{eq:condense}.  As we shall see in the next subsection, nested random sets, together with their corresponding contour functions \eqref{eq:contour} play an important role in this new theory.

\subsection{Predictive random sets}
\label{SS:prs}

In their investigations into the use of Dempster--Shafer theory for statistical inference, \citet{mzl2010} and \citet{zl2010} observe that the corresponding belief functions have proper calibration properties only for certain classes of assertions or hypotheses.  To rectify this mis-calibration, the previous authors argue that the Dempster--Shafer focal elements need to be enlarged, and that this can be accomplished by using what are called predictive random sets.  This combination of predictive random sets with the Dempster--Shafer theory of belief functions provides the mathematical backbone of a new approach, the so-called inferential model (IM) approach; for the complete details, see \citet{imbasics, imbasics.c}.  Here and in Section~\ref{S:plausibility}, we shall review this general theory with an emphasis on the construction of exact confidence regions. 

Given the auxiliary space $\UU$, equipped with a $\sigma$-algebra $\U$ and probability measure $\prob_U$, let $\SS$ be a collection of $\prob_U$-measurable subsets of $\UU$, assumed to contain $\varnothing$ and $\UU$.  The collection $\SS$ will serve as the support for the predictive random set.  To avoid measurability issues, assume that $\U$ contains all closed sets, and that all sets $S \in \SS$ are closed; without loss of generality, assume that $\SS$ contains both $\varnothing$ and $\UU$.  Write $\S$ for the predictive random set, $\prob_\S$ for its distribution, and $f_\S(u)$ for the corresponding contour function \eqref{eq:contour}.  An apparently new concept in the random set theory is that of validity.  That is, the predictive random set $\S$ is \emph{valid} if $f_\S(U)$ is stochastically no smaller than $\unif(0,1)$ when $U \sim \prob_U$.  It will be shown in Section~\ref{S:plausibility} that validity of the predictive random set leads to confidence regions with exact coverage probabilities. 

Here, the interesting question is how to construct a predictive random set that satisfies this validity criterion.  The answer is surprisingly simple.  Take $\S$ to be nested, so that its support $\SS$ satisfies \eqref{eq:nested}, and assume its plausibility function is condensable.  As discussed in the previous subsection, this implies that the contour function $f_\S$ fully characterizes the distribution $\prob_\S$.  Now, since validity implicitly requires some connection between $\prob_\S$ and $\prob_U$, our new condition should forge this connection.  Indeed, we shall consider $\S$ with contour functions $f_\S$ that satisfy
\begin{equation}
\label{eq:natural.measure}
f_\S(u) = 1-\sup_{S \in \SS: S \not\ni u} \prob_U(S), \quad u \in \UU. 
\end{equation}
We can now prove that these conditions are sufficient for validity.  

\begin{prop}
\label{prop:prs.valid}
Suppose $\S$ is nested with condensable plausibility function.  If its contour function satisfies \eqref{eq:natural.measure}, then it is valid.  
\end{prop}

\begin{proof}
Pick any $\alpha \in (0,1)$.  By \eqref{eq:natural.measure}, $f_\S(u) \leq \alpha$ if and only if there exists $S \in \SS$ with $\prob_U(S) \geq 1-\alpha$ and $S \not\ni u$.  Let $S_\alpha = \bigcap \{S \in \SS: \prob_U(S) \geq 1-\alpha\}$, the smallest $S$ with $\prob_U$-probability at least $1-\alpha$; then $f_\S(u) \leq \alpha$ if and only if $u \not\in S_\alpha$.  Since all sets in $\SS$ are closed, $S_\alpha$ is closed and, hence, $\prob_U$-measurable, with $\prob_U(S_\alpha) \geq 1-\alpha$.  Then $\prob_U\{f_\S(U) \leq \alpha\} = \prob_U(S_\alpha^c) \leq \alpha$.  This holds for all $\alpha$, so $f_\S(U)$ is stochastically no smaller than $\unif(0,1)$, and the claimed validity holds.  
\end{proof}

Nested predictive random sets are simple to construct.  For example, suppose $\prob_U$ is a $\unif(0,1)$ distribution and define a predictive random set $\S$ given by 
\begin{equation}
\label{eq:default.prs}
\S = \{u: |u-\tfrac12| \leq |U-\tfrac12|\}, \quad \text{with} \quad U \sim \prob_U. 
\end{equation}
Then the support $\SS$ of $\S$ contains all symmetric intervals $S$ centered at $\frac12$ of width less than or equal to 1, which is clearly a nested collection.  Also, note that the random set $\S$ is determined by a typical probability space $(\UU,\U,\prob_U)$ and a set-valued mapping, say, $\Gamma$, with $u \mapsto \Gamma(u) = \{u': |u'-\tfrac12| \leq |u-\tfrac12|\}$.  This matches the setup in \citet{dempster1967}, so condensability of the corresponding plausibility function follows from general theory; see \citet[][Sec.~V.F]{shafer1987} and the references therein.  For the contour function, 
\[ f_\S(u) = \prob_\S\{\S \ni u\} = \prob_U\{|U-\tfrac12| \geq |u-\tfrac12|\} = 1-|2u-1|. \]
Then condition \eqref{eq:natural.measure} clearly holds since the supremum is attained at $\Gamma(u)$, which has $\prob_U$-probability $|2u-1|$.  Therefore, $\S$ satisfies the conditions of Proposition~\ref{prop:prs.valid} and, hence, is valid.  However, validity of $\S$ can be shown directly by checking that its contour function above satisfies $f_\S(U) \sim \unif(0,1)$ for $U \sim \prob_U$.  \citet{imbasics} refer to \eqref{eq:default.prs} as the ``default'' predictive random set.  

The previous arguments can be generalized.  For example, let $\prob_U$ be a general non-atomic distribution on $\UU$ and take $h: \UU \to \RR$ to be a continuous function, constant only on $\prob_U$-null sets.  Then it follows similarly that the predictive random set $\S$, defined by 
\begin{equation}
\label{eq:gen.default.prs}
\S = \{u: h(u) \leq h(U)\}, \quad \text{with} \quad U \sim \prob_U, 
\end{equation}
is valid.  In many scalar parameter problems, by performing suitable auxiliary variable transformations, one can get $\UU=(0,1)$ and $\prob_U = \unif(0,1)$, so that the default predictive random set \eqref{eq:default.prs} can be used.  However, this more general construction of a valid predictive random set proves useful in cases where the auxiliary variable space $\UU$ is of dimension two or more; see Section~\ref{SS:lognormal}.

\section{Plausibility regions}
\label{S:plausibility}

\subsection{Plausibility functions for statistical inference}
\label{SS:inference}

Recall the auxiliary variable representation of the sampling model, i.e., $T=a(U,\theta)$, where $T$ is the statistic of interest, and $U \sim \prob_U$.  Let $T=t$ be the observed statistic.  If the auxiliary variable $U$ were also observed, say $U=u$, then the best possible inference on $\theta$ could be obtained, and would be represented by the set 
\[ \Theta_t(u) = \{\theta: t=a(u,\theta)\}. \]
This set could be a singleton, but need not be.  The idea is that \emph{if} the auxiliary variable were observed, then given $T=t$, one can solve for the parameter of interest, and $\Theta_t(u)$ is exactly this set of solutions.  In other words, $\Theta_t(u)$ defines a $t$-dependent compatibility relation \citep{shafer1987} on $\Theta$ and $\UU$.  

Since the auxiliary variable $U$ is not observable, it is not clear exactly how we should make use of the sets $\Theta_t(u)$.  In the classical Dempster--Shafer context, a belief function is defined on $\Theta$ by pushing the measure $\prob_U$ on $\UU$ forward through the $t$-dependent set-valued mapping $\Theta_t(\cdot)$, creating a new random set $\Theta_t(U)$, with $U \sim \prob_U$.  But as we indicated above in Section~\ref{SS:prs}, the Dempster--Shafer belief functions, generally, are not properly calibrated, and here is where the predictive random set $\S$ comes into play.  The validity property for $\S$ ensures that it will hit its target---a draw from $\prob_U$---with large $\prob_\S$-probability.  Therefore, we may push the measure $\prob_\S$, or its corresponding belief function forward, via the map $\Theta_t(\cdot)$, to obtain the bigger random set  
\begin{equation}
\label{eq:focal}
\Theta_t(\S) = \bigcup_{u \in \S} \Theta_t(u) 
\end{equation}
The intuition is that we expect $\Theta_t(\S)$ to contain the true $\theta$ with large $\prob_\S$-probability.  So, we understand $\Theta_t(\S)$ as a (random) set of ``reasonable'' guesses of $\theta$: for a given $A \subseteq \Theta$, if $\Theta_t(\S) \cap A \neq \varnothing$, then we cannot rule out the possibility that the true $\theta$ resides in $A$.  By the \emph{plausibility} of $A$ we mean the $\prob_\S$-probability that $\Theta_t(\S) \cap A \neq \varnothing$, 
\begin{equation}
\label{eq:plaus}
\pl_t(A) = \prob_\S\{\Theta_t(\S) \cap A \neq \varnothing\}, \quad A \subseteq \Theta.
\end{equation}
We shall refer to $\pl_t(A)$ as the plausibility function at $A$; though the notation does not reflect this, the reader should keep in mind that $\pl_t$ depends on $\prob_\S$.  

A few remarks about the above construction are in order.  First, we could have equivalently started by defining a belief function $\bel_t(A) = \prob_\S\{\Theta_t(\S) \subseteq A\}$ for the random set $\Theta_t(\S)$ and then the plausibility function $\pl_t(A) = 1-\bel_t(A^c)$.  We will not need the belief function in what follows, so the presentation here is more direct.  Second, the fact that the new random set in \eqref{eq:focal} is generally larger than that of the classical Dempster--Shafer analysis leads to smaller belief functions.  It is this squashing of the Dempster--Shafer belief function or, equivalently, the boosting of the plausibility function, that accounts for the improved calibration.  Indeed, as we show below, if the predictive random set is valid, then the squashing/boosting will be just enough to attain the desired calibration.  Third, the argument here for combining $\Theta_t(\cdot)$ with $\S$ as in \eqref{eq:focal} is just a special case of Dempster's rule of combination \citep{dempster1967, dempster2008}, though writing out the details formally perhaps does not provide any additional insight.  The key point is that Dempster's argument does not require that uncertainty on the $\UU$-space be summarized with a genuine probability measure.  In particular, the same line of reasoning applies if uncertainty on $\UU$ is described via a belief function, like in our present case.    

If $\prob_\S\{\Theta_t(\S) = \varnothing\} > 0$, then the formula \eqref{eq:plaus} must be adjusted.  To avoid such conditioning here, we assume that 
\begin{equation}
\label{eq:nonempty}
\Theta_t(u) \neq \varnothing \quad \text{for all $(t,u)$ pairs}.  
\end{equation}
This assumption essentially boils down to there being no non-trivial constraints on the parameter $\theta$ in the sampling model $\prob_{T|\theta}$.  An example of a non-trivial constraint is in a Poisson problem where the mean $\theta$ has a positive lower bound.  Most regular problems, including the examples in Section~\ref{S:examples}, satisfy \eqref{eq:nonempty}, though there are some that do not.  Assumption \eqref{eq:nonempty} is not necessary to construct plausibility regions with the desired coverage probabilities, but it will make our presentation easier.  The correction requires either conditioning on the event $\{\Theta_t(\S) \neq \varnothing\}$ or, preferably, stretching the predictive random set just enough to avoid the conflict \citep{leafliu2012}. 

For the important special case where $A=\{\theta\}$ is a singleton, we write $\pl_t(\theta) = \pl_t(\{\theta\}) = \prob_\S\{\Theta_t(\S) \ni \theta\}$.  Note that this special plausibility function is just the contour function \eqref{eq:contour} corresponding to the new nested random set $\Theta_t(\S)$.  This plausibility function also gives rise to the $100(1-\alpha)$\% \emph{plausibility region}:
\begin{equation}
\label{eq:plaus.region}
\P_\alpha(t) = \{\theta: \pl_t(\theta) > \alpha\}.  
\end{equation}
As we demonstrate below, if $\S$ is valid, then the plausibility region $\P_\alpha(T)$ has coverage probability at least $1-\alpha$ and, in many cases, equality is attained.  Perhaps more important than the coverage probability results is the fact that the plausibility regions are inferentially meaningful.  That is, for the given data, $\theta$ values inside the plausibility region are, in a certain sense, more plausible candidates for the true parameter value than $\theta$ values outside the plausibility region.

\subsection{Coverage probability results}
\label{SS:main}

The first result gives shows that $\pl_T(\theta)$, for $T \sim \prob_{T|\theta}$, is stochastically no smaller than $\unif(0,1)$ under mild conditions.  From this, plausibility region's advertised attainment of the nominal coverage coverage probability follows easily.  

\begin{prop}
\label{prop:valid}
Fix $\theta \in \Theta$.  If $\S$ is valid and \eqref{eq:nonempty} holds, then $\pl_T(\theta)$ is stochastically no smaller than $\unif(0,1)$ when $T \sim \prob_{T|\theta}$.  
\end{prop}

\begin{proof}
From the alternative description of the sampling model $\prob_{T|\theta}$ in \eqref{eq:sample}, for $T \sim \prob_{T|\theta}$, there exists a corresponding $U_T \sim \prob_U$ such that $T = a(\theta,U_T)$.  Moreover, it follows easily from the definition of $\Theta_t(\S)$ that $\Theta_T(\S) \ni \theta$ if and only if $\S \ni U_T$.  Therefore, $\pl_T(\theta) = \prob_\S\{\Theta_T(\S) \ni \theta\} = \prob_\S\{\S \ni U_T\} = f_\S(U_T)$.  Since $\S$ is valid, $f_\S(U_T)$ is stochastically no smaller than $\unif(0,1)$, as a function of $U_T \sim \prob_U$, and so the claim follows.  
\end{proof}

There are two relevant results that can be derived from Proposition~\ref{prop:valid} and its proof.
\begin{itemize}
\item The first is that, for any $\alpha \in (0,1)$, the coverage probability of the plausibility region $\P_\alpha(T)$ in \eqref{eq:plaus.region} is at least $1-\alpha$, i.e., $\prob_{T|\theta}\{\P_\alpha(T) \ni \theta\} \geq 1-\alpha$ for all $\theta$.  To see this, note that $\prob_{T|\theta}\{\P_\alpha(T) \ni \theta\} = \prob_{T|\theta}\{\pl_T(\theta) > \alpha\}$.  By Proposition~\ref{prop:valid}, $\pl_T(\theta)$ is stochastically no smaller than $\unif(0,1)$.  This implies that the latter probability is no smaller than $\prob\{\unif(0,1) > \alpha\} = 1-\alpha$, hence the claim.
\vspace{-2mm}
\item Second, confidence regions can be constructed directly from $\Theta_t(\cdot)$ and the support sets $S \in \SS$ for the predictive random sets.  Indeed, for fixed $S \in \SS$, we know from the above proof that $\Theta_T(S) \ni \theta$ if and only if $U_T \in S$.  So, $\prob_{T|\theta}\{\Theta_T(S) \ni \theta\} = \prob_U(S)$, and if we select $S$ such that $\prob_U(S) = 1-\alpha$, then $\Theta_t(S)$ is a $100(1-\alpha)$\% confidence region for $\theta$.  Therefore, an alternative $100(1-\alpha)$\% confidence region construction selects the smallest $S$ with $\prob_U(S) = 1-\alpha$ and takes $\C_\alpha(t) =\Theta_t(S)$.   
\end{itemize}

An important question is, under what conditions, is the coverage probability exactly equal to $1-\alpha$ or, equivalently, when is $\pl_T(\theta)$, with $T \sim \prob_{T|\theta}$, exactly uniformly distributed?  It turns out that there are two conditions needed.  First, $T$ must have a continuous distribution $\prob_{T|\theta}$, otherwise $\pl_T(\theta)$ cannot be continuous.  Second, the predictive random set must be \emph{exact}, not just valid, i.e., $f_\S(U)$ must be exactly $\unif(0,1)$ for $U \sim \prob_U$.  This exactness property holds for the default predictive random set \eqref{eq:default.prs} and its generalized version \eqref{eq:gen.default.prs}.  Therefore, for problems with continuous $T$, if we choose an exact predictive random set, such as one of those in \eqref{eq:default.prs} or \eqref{eq:gen.default.prs}, then the plausibility region $\P_\alpha(T)$ has coverage probability exactly $1-\alpha$.

\section{Examples}
\label{S:examples}

\ifthenelse{1=1}{}{
\subsection{Exponential mean}

Let $Y_1,\ldots,Y_n$ be an iid sample from an exponential distribution with probability density function $p_\theta(x) = (1/\theta)e^{-x/\theta}$, $x \geq 0$, $\theta > 0$; in this parametrization, $\theta$ represents the mean of the distribution.  Here, $\theta$ is unknown and the goal is to construct a plausibility interval for $\theta$ based on the observable $X_1,\ldots,X_n$.  As a starting point, set $T=\sum_{i=1}^n X_i$, the minimal sufficient statistic for $\theta$.  Since $T$ has a gamma distribution with shape $n$ and scale $\theta$, an association is obtained by writing 
\[ T = F_{n,\theta}^{-1}(U), \quad U \sim \unif(0,1), \]
where $F_{n,\theta}$ is the corresponding gamma distribution function.  If we select the default predictive random set $\S$ in \eqref{eq:default.prs} for $U$, then, given $T=t$, the plausibility function can be readily written down in terms of this gamma distribution function, i.e., 
\[ \pl_t(\theta) = 1-|2 F_{n,\theta}(t)-1|, \quad \theta > 0. \]
Then the $100(1-\alpha)$\% plausibility interval $\P_\alpha(t)$ is given by 
\[ \{\theta: 1-|2 F_{n,\theta}(t)-1| > \alpha\} = \{\theta: \alpha/2 < F_{n,\theta}(t) < 1-\alpha/2\}. \]
Since $\theta$ is a scale parameter in $F_{n,\theta}$, i.e., $F_{n,\theta}(t) = F_{n,1}(t/\theta)$, $\P_\alpha(t)$ can be written as a genuine interval:
\[ \P_\alpha(t) = \Bigl( \frac{t}{\gamma_{n,1}(1-\frac{\alpha}{2})}, \, \frac{t}{\gamma_{n,1}(\frac{\alpha}{2})} \Bigr), \]
where $\gamma_{n,1}(q) = F_{n,1}^{-1}(q)$ is the $q$th quantile of the gamma distribution with shape $n$ and scale 1.  It follows from the general theory that $\P_\alpha(t)$ has exact coverage probability $1-\alpha$.  This interval also happens to be the confidence interval obtained by inverting the exact size-$\alpha$ likelihood ratio test. 
}

\subsection{Power law process}
\label{SS:power}

Consider a continuous time non-homogenous Poisson process $\{N_y: y \geq 0\}$, where the mean function $m(y) = \E(N_y)$ satisfies $m(y) = \psi y^\theta$, for $\psi, \theta > 0$.  Such a process is called a power law process \citep[e.g.,][]{gaudoin.yang.xie.2006}.  The parameter $\psi$ is a scale parameter and $\theta$ is a shape parameter.   Though both $\psi$ and $\theta$ are unknown, the goal here is to construct a plausibility interval for $\theta$ based on $n$ observed event times $Y_1 \leq \cdots \leq Y_n$.  

For these data, the log-likelihood function for $(\psi,\theta)$ looks like 
\[ \ell(\psi,\theta) = n\log\psi + n\log\theta + (\theta-1) \sum_{i=1}^n \log Y_i - \psi Y_n^\theta. \]
By the Neyman--Fisher factorization theorem, a joint sufficient statistic for $(\psi,\theta)$ is the pair $(\sum_{i=1}^n \log Y_i, Y_n)$.  This sufficient statistic is a one-to-one transformation of the maximum likelihood estimator $(\hat\psi,\hat\theta)$, given by 
\[ \hat\theta = \frac{n}{\sum_{i=1}^{n-1} \log(Y_n / Y_i)} \quad \text{and} \quad \hat\psi = n / Y_n^{\hat\theta}. \]
Therefore, $(\hat\psi,\hat\theta)$ is a minimal sufficient statistic for $(\psi, \theta)$.  Moreover, for all $\psi$, the vector $\{\log(Y_n/Y_{n-1}), \ldots, \log(Y_n / Y_1) \}$ is distributed as a sorted sample of $n-1$ independent random variables from an exponential distribution with mean $1/\theta$ \citep[e.g.,][]{crow1974}.  For simplicity, we take $T = n / \hat\theta$, which has a gamma distribution with shape $n-1$ and scale $1/\theta$.  Some information about $\theta$ is lost by ignoring the $\hat\psi$ component of the joint minimal sufficient statistic, but the required marginalization strategy is beyond our present scope; see \citet{immarg}.  So, we shall consider here the simple association 
\[ T = F_{n-1,1/\theta}^{-1}(U), \quad U \sim \unif(0,1), \]
where $F_{n-1,1/\theta}$ denotes the gamma distribution function with shape $n-1$ and scale $1/\theta$.  If we use the default predictive random set $\S$ in \eqref{eq:default.prs}, then the plausibility function is
\[ \pl_t(\theta) = 1-|2F_{n-1,1/\theta}(t) - 1|, \quad \theta > 0, \]
which can be readily evaluated numerically.  Then the $100(1-\alpha)$\% plausibility interval $\P_\alpha(t)$ for $\theta$ is given by 
\[ \P_\alpha(t) = \{\theta: \pl_t(\theta) > \alpha\} = \{\theta: \alpha/2 < F_{n-1,1/\theta}(t) < 1-\alpha/2\}. \] 
Since $1/\theta$ is a scale parameter in $F_{n-1,1/\theta}(t)$, the right-hand side above can be rewritten as $\{\theta: \alpha/2 < F_{n-1,1}(\theta t) < 1-\alpha/2\}$.  Therefore, if we let $\gamma_{n-1,1}(q)$ denote the $q$th quantile of the gamma distribution with shape $n-1$ and scale 1, then the plausibility interval can be written as a genuine interval, 
\[ \P_\alpha(t) = \Bigl( \frac{\gamma_{n-1,1}(\frac{\alpha}{2})}{t}, \, \frac{\gamma_{n-1,1}(1-\frac{\alpha}{2})}{t} \Bigr). \]
This is equivalent to the exact confidence interval given in equation (6) of \citet{gaudoin.yang.xie.2006} in terms of chi-square quantiles.

\subsection{Exponential regression through the origin}
\label{SS:reg}

Consider a special case of an exponential log-linear model, where $Y_1,\ldots,Y_n$ are independent exponential random variables and $Y_i$ has mean $e^{\theta x_i}$, $i=1,\ldots,n$, for fixed covariates $x_1,\ldots,x_n$.   The goal is to produce a plausibility interval for the slope parameter $\theta$.  

The log-likelihood function for $\theta$ looks like 
\[ \ell(\theta) = -\sum_{i=1}^n (\theta x_i + e^{\log Y_i - \theta x_i} ), \]
and the likelihood equation is given $\sum_{i=1}^n (e^{\log Y_i - \theta x_i} - 1) x_i = 0$.  Let $T$ be the solution to this equation, the maximum likelihood estimator of $\theta$.  If $G_\theta = G_{\theta, x, n}$ is the distribution function of $T$, then a suitable association is, again, given by 
\[ T = G_\theta^{-1}(U), \quad U \sim \unif(0,1). \]
The distribution function $G_\theta$ is not available in closed form, but it can be evaluated via Monte Carlo.  If the default predictive random set $\S$ in \eqref{eq:default.prs} is used for $U$, then again the plausibility function for $\theta$ is 
\[ \pl_t(\theta) = 1-|2 G_\theta(t) - 1|, \quad \theta \in \RR. \]
No expressions are available for the plausibility function in this case, but, again, it is relatively easy to evaluate numerically via Monte Carlo.  

For illustration, Figure~\ref{fig:reg} displays plots of the plausibility function $\pl_t(\theta)$, as a function of $\theta$, for two simulated data sets, one of size $n=10$, the other of size $n=20$.  Here the covariate values are $x_i=i$, $i=1,\ldots,n$, and the true parameter value is $\theta=1$.  This function is evaluated by a Monte Carlo integration step performed at each point $\theta$ on the horizontal axis.  For comparison, the endpoints of the 95\% confidence interval based on asymptotic normality of the maximum likelihood estimator are also displayed.  In both cases, the two intervals are comparable, which is to be expected.  However, for such small $n$, it is unlikely that the asymptotic normality has kicked in, so the actual coverage probability of the latter is likely different from the target 0.95.  The plausibility interval, on the other hand, has coverage probability exactly equal to 0.95 based on the theory developed in Section~\ref{SS:main}.  Indeed, in a simulation of 5000 data sets of size $n=10$, under same setup as above, the estimated coverage probabilities for the exact plausibility interval and asymptotic confidence interval are 0.951 and 0.934, respectively.  

\ifthenelse{1=1}{}{
> er.pl.sim(1, 10, 2500, 5000)
      Coverage   Length
Plaus   0.9510 0.000000
Likhd   0.9344 0.201156
}

\begin{figure}
\begin{center}
\subfigure[$n=10$]{\scalebox{0.6}{\includegraphics{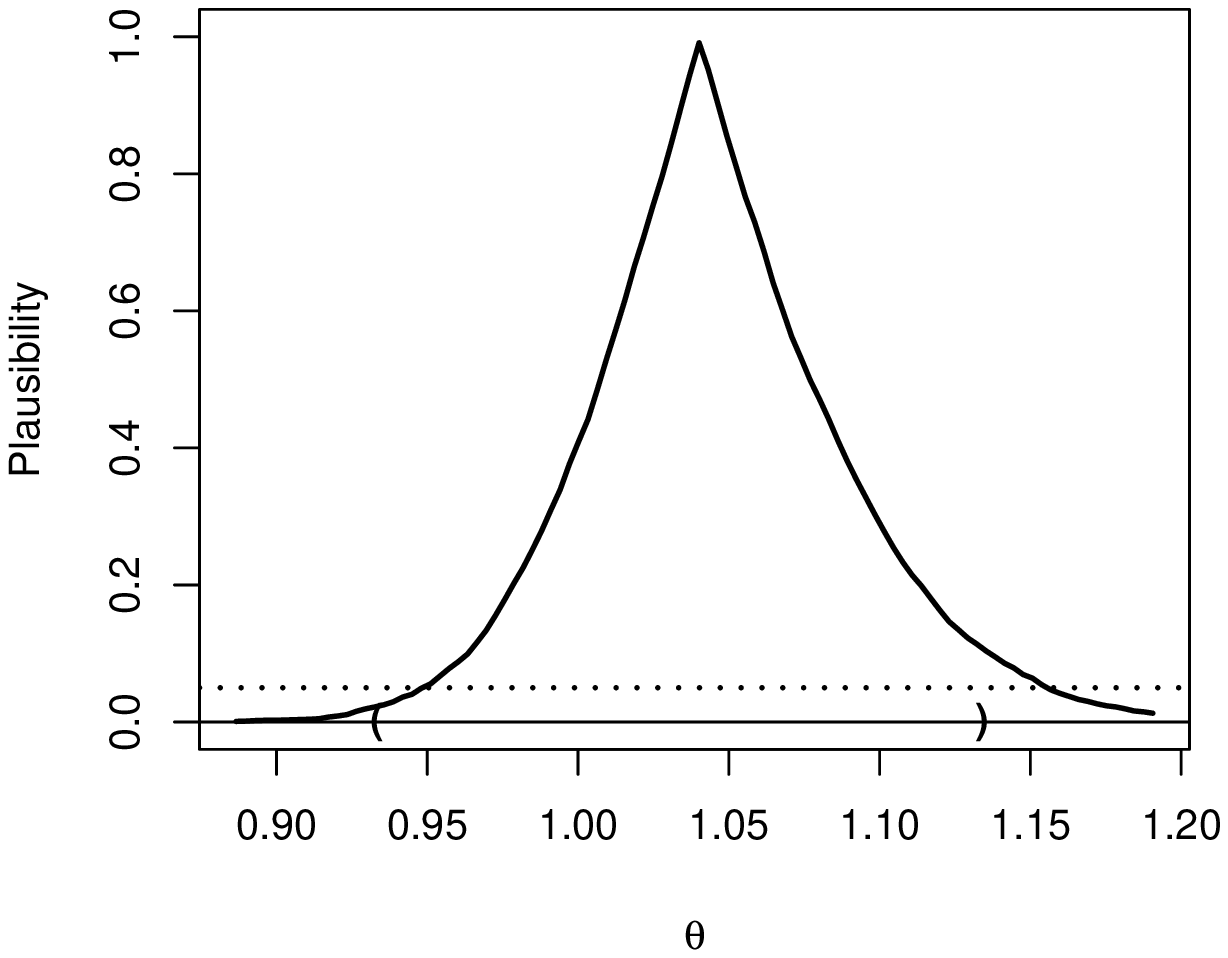}}}
\subfigure[$n=20$]{\scalebox{0.6}{\includegraphics{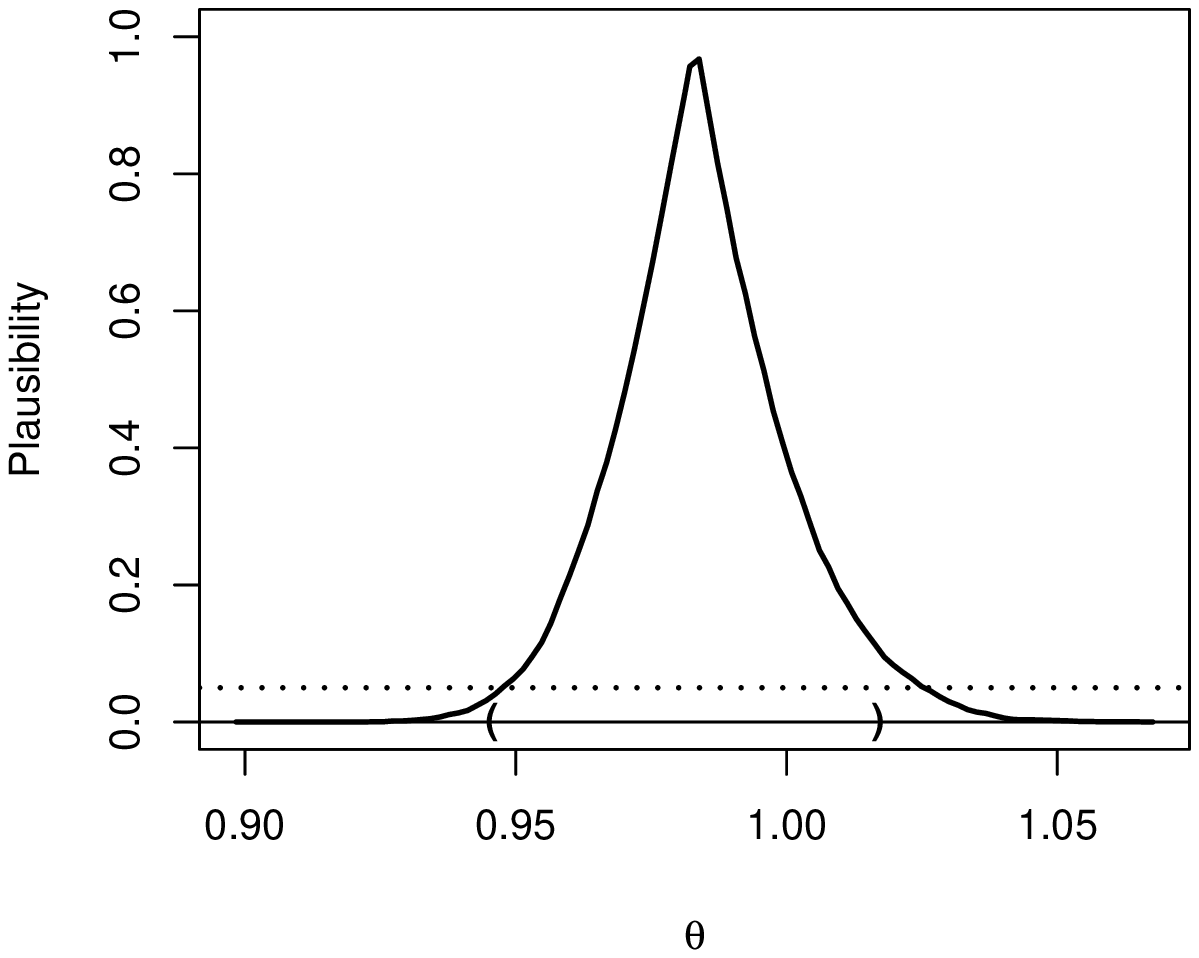}}}
\end{center}
\caption{Plausibility functions $\pl_t(\theta;\S)$ versus $\theta$ in Section~\ref{SS:reg}.  Parentheses on the $\theta$-axis mark the endpoints of the 95\% confidence interval based on asymptotic normality of the maximum likelihood estimator.  Horizontal line at $\alpha=0.05$ determines the endpoints of the 95\% plausibility interval.}
\label{fig:reg}
\end{figure}

\subsection{Lognormal model}
\label{SS:lognormal}

A useful two-parameter model for lifetime data is lognormal, i.e., a normal distribution is used to model the log lifetimes.  Here, for simplicity, we will work directly with the transformed data.  That is, the starting point is an independent sample $Y_1,\ldots,Y_n$ from a normal distribution $\nm(\mu,\sigma^2)$, where $Y_i$ is the log of the $i$th lifetime, and the goal is inference on $\theta=(\mu,\sigma^2)$.  The analysis that follows can be extended to the regression case without too much difficulty; the only challenge is dealing with dependence introduced by non-orthogonal predictors.  Some remarks on how the normality assumption can be removed are given at the end of this subsection.  

This simple model admits a well known joint minimal sufficient statistic, namely, $T=(T_1,T_2)=(\bar Y, \sum_{i=1}^n (Y_i-\bar Y)^2)$ and, moreover, the two entries are statistically independent, with $T_1 \sim \nm(\mu, \sigma^2/n)$ and $T_2 \sim \sigma^2 \chisq(n-1)$.  Let $U=(U_1,U_2)$ be a pair of independent $\unif(0,1)$ auxiliary variables, and write the relation $T=a(U,\theta)$ in \eqref{eq:sample} as  
\[ T_1 = \mu + \sigma n^{-1/2} F_1^{-1}(U_1), \quad T_2 = \sigma^2 F_2^{-1}(U_2), \]
where $F_1$ and $F_2$ are the $\nm(0,1)$ and $\chisq(n-1)$ distribution functions, respectively.  For the predictive random set $\S$, we propose the box:
\[ \S=\{(u_1,u_2): \max(|u_1-\tfrac12|,|u_2-\tfrac12|) \leq \max(|U_1-\tfrac12|,|U_2-\tfrac12|)\}, \]
where $U_1,U_2$ are independent $\unif(0,1)$.  This is a two-dimensional version of the default predictive random set in \eqref{eq:default.prs} for a single uniform auxiliary variable, and also a special case of \eqref{eq:gen.default.prs} with $h(u_1,u_2) = \max(|u_1-\tfrac12|,|u_2-\tfrac12|)$.  Note that one should not take the predictive random set $\S$ for the pair $(U_1,U_2)$ to be the Cartesian product of two independent default predictive random sets---the support of the latter will not be nested in the sense \eqref{eq:nested}.  Now, for given $T=t$, it is straightforward to write down the plausibility function for $\theta=(\mu,\sigma^2)$:
\[ \pl_t(\mu,\sigma^2) = 1 - \max\Bigl\{ \Bigl|2 F_1\Bigl(\frac{t_1-\mu}{\sigma n^{-1/2}} \Bigr)-1 \Bigr|, \Bigl|2 F_2\Bigl(\frac{t_2}{\sigma^2}\Bigr)-1 \Bigr| \Bigr\}^2. \]
Then the $100(1-\alpha)$\% plausibility region is the just the $\alpha$-level set for $\pl_t$.  

For illustration, Figure~\ref{fig:lnpl} plots the corresponding 90\% confidence regions for three methods: the plausibility region described above, the elliptical region based on the asymptotic normality of the maximum likelihood estimator of $(\mu,\sigma^2)$, and a naive confidence region obtained by taking the Cartesian product of standard confidence intervals for $\mu$ and $\sigma^2$ individually, with Bonferroni correction.  This is based on an independent sample of size $n=25$ from $\nm(0,1)$.  The plausibility region has an unusual trapezoidal shape, and appears to be larger than the maximum likelihood regions.  To check the efficiency, a small simulation was conducted, based on 5000 Monte Carlo samples, and we found that the coverage probabilities for the 90\% plausibility, maximum likelihood, and naive confidence regions were 0.899, 0.841, and 0.904, respectively.  So, apparently, the maximum likelihood region's coverage probability is well below the target 0.90.  

\begin{figure}
\begin{center}
\scalebox{0.6}{\includegraphics{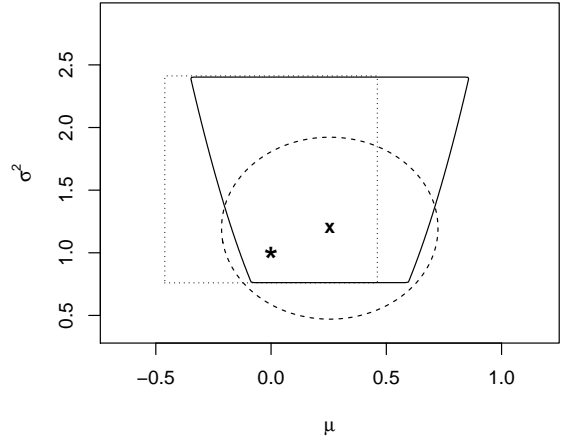}}
\end{center}
\caption{90\% confidence regions for $(\mu,\sigma^2)$ in the lognormal example in Section~\ref{SS:lognormal}.  Solid line is the plausibility region, dashed line is the maximum likelihood region, and dotted line is the naive Cartesian product region.  Asterisk marks the true $(\mu,\sigma^2)$ and {\sf x} marks the maximum likelihood estimate.}
\label{fig:lnpl}
\end{figure}

What makes the normal, or lognormal, distribution special in this case is the initial reduction to a two-dimensional joint minimal sufficient statistic.  Many other distributions do not admit such a reduction.  For example, suppose $Y_1,\ldots,Y_n$ are independent, with distribution function $F(\frac{y-\mu}{\sigma})$, where $\mu$ and $\sigma$ are location and scale parameters, respectively.  In this context of lifetime distributions, $Y_i$ might be the log lifetime of the  $i$th observation and $F$ could be a logistic distribution function; in this case, the lifetime distribution is log logistic.  But the logistic distribution does not admit a dimension reduction via sufficiency.  However, one can always write the model as 
\[ Y_i = \mu + \sigma F^{-1}(U_i), \quad i=1,\ldots,n, \]
where $U_1,\ldots,U_n$ are independent $\unif(0,1)$, and follow the procedure outlined above, with the box predictive random set, to get a plausibility function 
\[ \pl_y(\mu,\sigma) = 1 - \Bigl\{ \max_{1 \leq i \leq n} \Bigl|2F\Bigl(\frac{y_i-\mu}{\sigma}\Bigr) - 1 \Bigr| \Bigr\}^n. \]
Unlike the normal case above, the plausibility region in this case will not be efficient, due to the fact that there are $n$ auxiliary variables being predicted for only two parameters.  Some reduction would be possible via conditioning, but the details would be rather involves; see \citet{imcond}.

\section{Discussion}
\label{S:discuss}

In this paper, we discuss a new approach for the construction of confidence regions based on the theory of random sets.  The key result is that if the predictive random set $\S$ for the unobservable auxiliary variable $U$ is valid, in the sense that it misses its its target not too often, then the corresponding plausibility region has at least the nominal coverage probability.  It is important that this validity result is not asymptotic and, moreover, does not depend on any characteristic of the problem that is unknown.  Therefore, it is generally quite easy to specify a valid predictive random set, and a default choice is given here and used in several examples to obtain practically useful results.  

Here the focus was on simplicity rather than generality.  Though the two examples involved only scalar parameter, essentially the same strategy would apply for a multi-parameter problem.  A challenging problem in multi-parameter situations is to give an exact confidence region for some component or, more generally, some scalar-valued function of the full parameter.  This was the actual setup in the power-law process example in Section~\ref{SS:power}, though we sidestepped the main difficulty by ignoring a part of the minimal sufficient statistic.  To incorporate all the information in the minimal sufficient statistic requires some careful manipulations which were beyond the present scope.  A new and detailed look at such problems is given \citet{immarg}.  

The primary goal here was to construct confidence regions that attain the nominal coverage probability.  We found that, in many cases, including the two examples in Section~\ref{S:examples}, the plausibility regions will actually hit this target on the nose.  A natural follow-up question is if these plausibility regions are ``optimal'' in some sense, i.e., do the plausibility regions have smallest average size, say, among all those regions that hit the desired coverage probability?  This question is the focus of ongoing investigations.

\section*{Acknowledgments}

The author is thankful for comments and suggestions given by Chuanhai Liu and an anonymous referee.  This work is partially supported by the U.S.~National Science Foundation, DMS--1208833.

\bibliographystyle{apa}
\bibliography{/Users/rgmartin/Dropbox/Research/mybib}

\end{document}